\documentclass[11pt]{amsart} 

   \usepackage[a4paper, margin=100pt]{geometry}     
     \normalsize 
    \addtocontents{toc}{\setcounter{tocdepth}{2}} 
  
    \usepackage{xcolor}
    \usepackage{hyperref}
         \hypersetup{colorlinks, citecolor=blue, filecolor=blue, linkcolor=black, urlcolor=blue}
    \usepackage{soul}
    \usepackage[numbers]{natbib}    
    \usepackage{amssymb}
    \usepackage{tikz-cd}

\usepackage{marginnote}

\usepackage{soul} 

\DeclareMathOperator{\GL}{\rm{GL}}
\DeclareMathOperator{\SL}{\rm{SL}}
\DeclareMathOperator{\Aff}{\rm{Aff}}

\newcommand{\R}{\mathbb{R}}\newcommand{\N}{\mathbb{N}}

\newcommand{\C}{\mathbb{C}}

\newcommand{\id}{\mathrm{Id}}

\newcommand{\dimHD}{\operatorname{dim_{H}}}

\newcommand{\dimUB}{\operatorname{\overline{dim}_{B}}}
\newcommand{\dimAff}{\operatorname{dim_{Aff}}}

\newcommand{\ua}{{\underline{a}}}

\newcommand{\cC}{\mathcal{C}}

\newcommand{\cF}{\mathcal{F}}

\newcommand{\cH}{\mathcal{H}}

\newcommand{\cU}{\mathcal{U}}

\newcommand{\cW}{\mathcal{W}}

\newcommand{\bF}{\mathbb{F}}

\newcommand{\bR}{\mathbb{R}}

\numberwithin{equation}{section}

\newtheorem{theorem}{Theorem}[section] 

\newtheorem{lemma}[theorem]{Lemma}
\newtheorem*{lemma*}{Lemma}
\newtheorem{proposition}[theorem]{Proposition}
\newtheorem*{proposition*}{Proposition}
\newtheorem{conjecture}[theorem]{Conjecture}

\newtheorem*{question*}{Question}
\newtheorem*{theorem*}{Theorem}
\newtheorem*{claim*}{Claim}
\newtheorem*{conjecture*}{Conjecture}

\newtheorem{theoremain}{Theorem}

\theoremstyle{definition}
\newtheorem{definition}[theorem]{Definition}

\theoremstyle{remark}
\newtheorem{remark}[theorem]{Remark}

    \usepackage{bigints}
    \usepackage{xfrac}
    \usepackage{mathrsfs}
    \usepackage{comment}
    \usepackage{verbatim}
    \usepackage{adjustbox}
    \usepackage{graphicx}
    \usepackage{import}

    \newcommand{\sldr}{\mathrm{SL}(d,\R)}
    \newcommand{\gldr}{\mathrm{GL}(d,\R)}
    
    \newcommand{\matrixexp}{\mathfrak{exp}}

\addtocontents{toc}{\setcounter{tocdepth}{1}}

\begin{document}
\title[Cantor sets in higher dimensions II]
{Cantor sets in higher dimensions II: 
Optimal dimension constraint for stable intersections
}

\author{Meysam Nassiri}
\address{School of Mathematics, Institute for Research in Fundamental Sciences (IPM), P.O. Box 19395-5746, Tehran, Iran}
\email{nassiri@ipm.ir}
\author{Mojtaba Zareh Bidaki}
\address{School of Mathematics, Institute for Research in Fundamental Sciences (IPM), P.O. Box 19395-5746, Tehran, Iran}
\email{mojtabazare@ipm.ir}


\begin{abstract} 
It is well known that a pair of compact sets in $\mathbb{R}^d$ ($d \in \mathbb{N}$) can be separated by small deformations if the sum of their upper box dimensions is less than $d$. In this paper, we demonstrate that this dimension constraint is optimal for regular Cantor sets. Specifically, for any prescribed upper box dimensions whose sum is greater than $d$, we construct classes of pairs of regular Cantor sets that exhibit $C^{1+\alpha}$-stable intersections.

Our method is geometrically flexible, enabling the construction of examples with arbitrarily small thickness in both projectively hyperbolic and nearly conformal regimes. These results also extend to the complex setting for holomorphic Cantor sets in $\mathbb{C}^d$. The proof relies on the ``covering criterion" for stable intersection introduced in the first part of this series \cite{NZ1}, which generalizes the ``recurrent compact set criterion" of Moreira--Yoccoz to higher dimensions.
\end{abstract}

\maketitle
\tableofcontents

\section{Introduction}

The arithmetic difference set
\[
K-K' := \{k-k' : k \in K,\; k' \in K'\}
\]
of two compact sets \(K, K' \subset \mathbb{R}^d\) is a central object in fractal geometry and dynamical systems. A fundamental problem in this area is to identify conditions ensuring that \(K-K'\) has positive Lebesgue measure (or has non-empty interior).

It is well established that for compact sets satisfying the dimensional deficit
\[
\dimUB(K) + \dimUB(K') < d,
\]
the set \(K-K'\) has zero Lebesgue measure. Consequently, almost all translations $K'+t$ are disjoint from $K$. Here, $\dimUB$ denotes the upeer box dimension (see \S \ref{sec: pre}).

The regime where \(\dimUB (K) + \dimUB (K') > d\) is markedly different and considerably more subtle. 
Here, one expects not only positive measure but potentially the presence of interior points. 
This aligns with the Palis conjecture \cite{Palis-conj-cantors} regarding the arithmetic sum of regular Cantor sets on the real line. For such sets (including self-affine sets) one considers the stronger notion of \emph{stable intersection}. This property implies that not only does \(K-K'\) possess a non-empty interior, but that this intersection persists robustly under small perturbations of the defining dynamical systems.
Recall that a regular Cantor set, by definition, is the unique invariant set under the action of expanding smooth maps with restricted to a specific Markov partition. Two regular Cantor sets are $C^r$-close if their generating maps and Markov partitions are close in the $C^r$ topology and Hausdorff topology, respectively. For precise definitions we refer to 
\cite{NZ1}.

While this problem is completely understood in one dimension ($d=1$), where the Palis conjecture has been established in both the real \cite{MY} and complex \cite{AMZ25} settings, the higher dimensional case ($d\ge 2$) remains largely unexplored. A primary question is whether the dimension constraint is optimal for all $d\in\N$. 
The aim of this paper is to answer this question affirmatively.

\begin{theoremain}
\label{thm: main1}
Let $d \in \mathbb{N}$  and $p,q  \in (0,d)$ with $p+q>d$. Then there exists a pair of regular Cantor sets $(K,K')$ in $\mathbb{R}^d$ with $\dimUB(K)= p$ and $\dimUB(K')=q$ such that the pair $(K,K')$ has $\cC^{1+\alpha}$-stable intersection ($\alpha>0$).
\end{theoremain}

Several distinct methods have been developed to verify stable intersections since the late 1960s. In the case \( d = 1 \), the classical \emph{gap lemma}, introduced by Newhouse \cite{n_1}, provides an elegant sufficient condition known as the \emph{thickness test}, although its applicability is limited (see also \cite{BieblerGapLemma} for the one-dimensional complex setting). However, its recent generalization to higher dimensions \cite{Yav22}, even under additional geometric assumptions, does not yield stable intersections.
In cases where both Cantor sets have sufficiently large dimensions, namely $\lfloor \dimHD K \rfloor + \lfloor \dimHD K' \rfloor \geq d$ where $\lfloor x \rfloor$ denotes the integer part of $x$, a fundamentally different approach was carried out by Asaoka \cite{Asaoka}. In this regime, the stable geometric properties of blenders 
allow one to produce examples of two ``transversal'' blender-type Cantor sets with $\cC^1$-stable intersections.  
However, as illustrated in Figure \ref{fig:dimension diagram}, a significant gap remains
in the literature for sets with smaller individual dimensions.

\begin{figure}[ht]
    \centering
    \includegraphics[width=.6\linewidth]{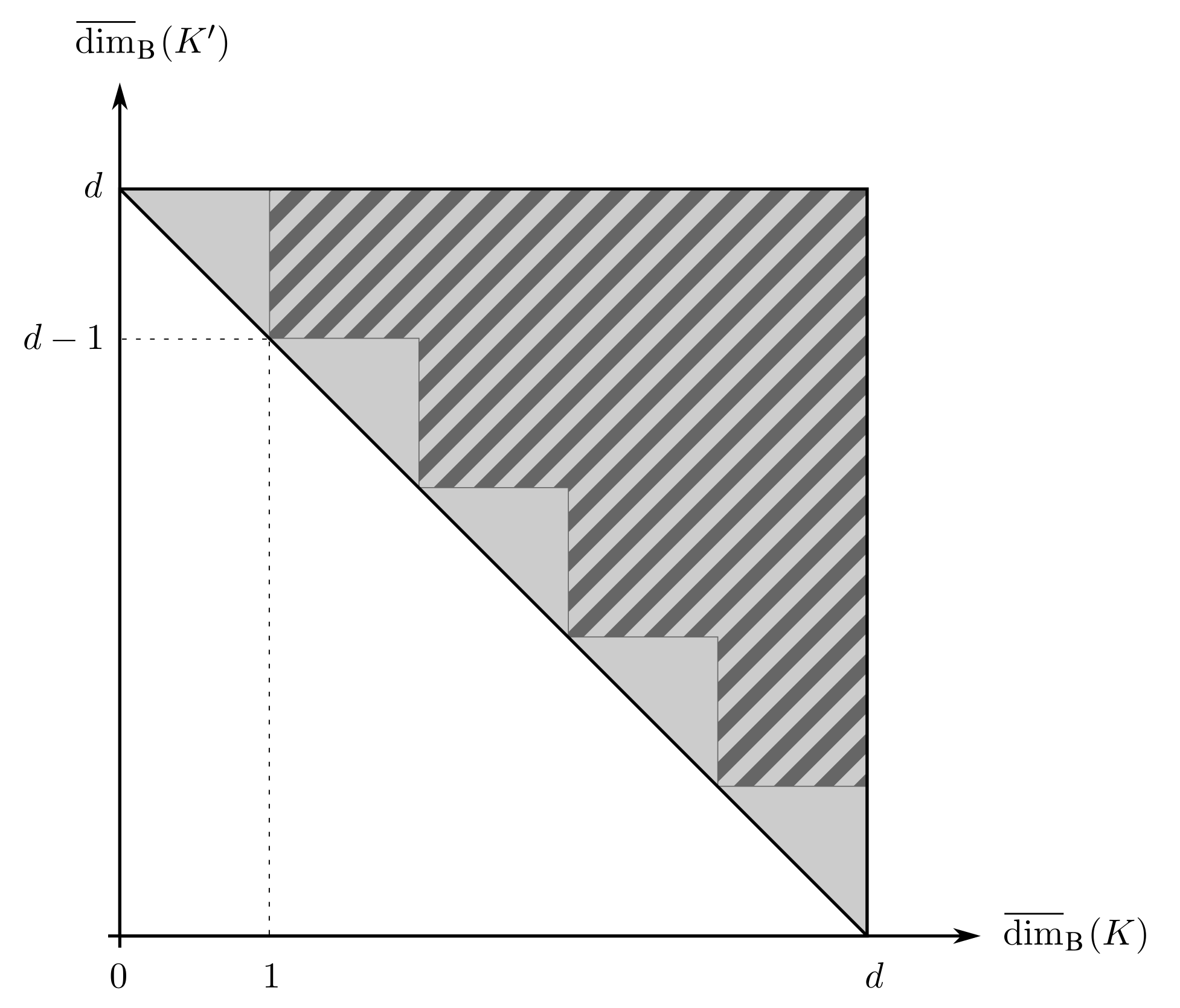}
    \caption{ Dimension pairs \((\dimUB(K), \dimUB(K'))\). Shaded region: pairs previously achieved by blender-type Cantor sets. Light gray region: new dimension pairs accessed via Theorems \ref{thm: main1} and \ref{thm: main2}.}
    \label{fig:dimension diagram}
\end{figure}

The proof of Theorem \ref{thm: main1} is based on the renormalization method. 
 It relies on the ``covering criterion'' introduced in the first paper of this series 
 \cite{NZ1}, where we generalize the ``recurrent compact set criterion'' of Moreira-Yoccoz \cite{MY} to higher dimensions.
A simplified version of the covering criterion for affine Cantor sets (or so-called self-affine sets) is stated in Theorem \ref{thm: covering for affines}. 

This allows us to develop a method to obtain ample classes of examples showing several distinct  geometrical features. 

\begin{theoremain}\label{thm: main2} 
The Cantor sets in Theorem \ref{thm: main1}  can be chosen to be affine, with arbitrarily small thickness, and nearly homothetic to an arbitrary element of $\sldr$.  
Furthermore, the construction can be adapted to the complex setting for holomorphic Cantor sets in $\C^d$. 
\end{theoremain}

Here, a regular Cantor $K$ set is (\emph{nearly homothetic}, resp.) \emph{homothetic} to a matrix $L\in \sldr$ if the  derivative of its generating map at every point $p\in K$ is equal to (close to, resp.)  $\lambda(p)\cdot  L$ for some $\lambda(p)>0$. 

Theorem \ref{thm: main2} demonstrates that the covering criterion can be applied to pairs of Cantor sets generated by expanding maps that share a common invariant (strong) unstable cone field. This, in particular, shows how the criterion operates for the class of Cantor sets that are far from being quasi-conformal; 
{that is, it remains effective even in the presence of strong affine distortion and anisotropy.}

The techniques developed herein suggest that the covering criterion is sufficiently powerful to address the stable intersection problem in the critical case where $K=K'$. We formulate this in the following conjecture.

\begin{conjecture}\label{conj:Falconer-regularCantor}
    Every regular Cantor set 
    \(K \subset \mathbb{R}^d\) with \(\dimHD(K) > d/2\) can be 
    $C^\infty$-approximated by a regular Cantor set $\tilde K$ such that the pair \((\tilde{K}, \tilde{K})\) has $C^r$-stable intersection for some $r>1$.  
\end{conjecture}

This conjecture can be viewed as a ``stable'' version of Falconer’s distance set conjecture \cite{FalconerConj} adapted to regular Cantor sets.
We remark that Falconer’s conjecture is still open even for $d=2$, for which the best-known dimension threshold is $5/4$,  obtained in \cite{GIOW20}. In forthcoming work, we intend to address the Conjecture \ref{conj:Falconer-regularCantor} for certain classes of bunched regular Cantor sets.

\subsection*{Acknowledgment} 
The authors are grateful for fruitful conversations with C. G. Moreira and M. Pourbarat. 
This work was completed while the authors were visiting BICMR,  Peking University.

\section{Preliminaries}\label{sec: pre}

We follow the notation and terminology of \cite{NZ1}. For the reader’s convenience, we briefly recall the relevant definitions and conventions.

\begin{itemize}
\item We denote the Hausdorff dimension of the set $A\subset \R^d$ as $\dimHD(A)$.
\item 
We denote the upper box dimension of the set $A \subset \mathbb{R}^d$ as$$\dimUB(A) := \limsup_{\varepsilon \to 0} \frac{\log N(A, \varepsilon)}{\log(1/\varepsilon)},$$
where $N(A, \varepsilon)$ is the minimum number of sets of diameter $\varepsilon$ required to cover $A$.

\item 
Given a metric space $(X,d)$, for any $x\in X$ we denote the $\delta$-neighborhood of $x$ in $X$ as $B_{\delta}(x)$. For $V\subset X$ we define 
$B_{\delta}(V):=\bigcup_{v\in V}B_{\delta}(v)$ and
$V_{(\delta)}:=\{v\in V:\; B_{\delta}(v)\subset V \}$.

\item 
Denote the space of invertible matrices over $\bR^d$ by $\GL(d,\bR)$, those with determinant equal to $1$ as $\SL(d,\bR)$ and orthonormal matrices by $\operatorname{SO}(d)$. For $A\in \GL(d,\bR)$ we denote its norm by $\|A\|_{\it op}:= \sup_{|v|=1} |Av|$ and its co-norm by $m(A):=\inf_{|v|=1} |Av|$.
\item We denote the identity matrix in $\GL(d,\bR)$ as $\id$. We also use the notation $\id$ for the identity function. Each case will be clear in the context.
\item 
We denote  by $\Aff(d,\bF)$ the space of invertible affine transformations $[x\mapsto Ax+a]$ of $\bF^d$ equipped with $\cC^{1}$ topology,  where $A\in \GL(d,\bF)$ and $a\in \bR^d$. We denote the subspace $\Aff_{\id}(d,\R)\subset \Aff(d,\R)$ as the space of affine transformations of the form 
$[x\mapsto \lambda \cdot x+ v]$ where $\lambda\in \R\setminus \{0\}$ and $v\in \R^d$.
\end{itemize}

\begin{definition}[$\cC^{1+\alpha}$ neighborhood of a regular Cantor set]
    Given a regular Cantor set $K$ generated by $g$ with symbolic type $\Sigma$, we denote $U^{1+\alpha}_{K,\delta}$ as a $\delta$-neighborhood in $\cC^{1+\alpha}$ topology of $K$ in the space of regular Cantor sets consists of Cantor sets $\tilde K$ with symbolic type $\Sigma$ generated by some $\cC^{1+\alpha}$ map $\tilde g$ with pieces $\tilde G(a)$ with the property that $g'|_{G'^*(a)}$ is $\delta$-close to $g|_{G^*(a)}$  in $\cC^{1+\alpha}$ topology, for all $a\in \mathfrak{A}$.
\end{definition}

\begin{definition}[Stable intersection] Given a pair of regular Cantor sets $K,K'\subset \R^d$, we say that $K,K'$ are $\cC^{1+\alpha}$-stably intersecting if there exists $\varepsilon>0$ such that for any $\tilde K\in U^{1+\alpha}_{K,\varepsilon}$ and $\tilde K'\in U^{1+\alpha}_{K',\varepsilon}$ one has $\tilde K\cap \tilde K'\neq \emptyset$.
\end{definition}

\begin{definition}[Bunched Cantor sets] \label{def bunched}
Given a regular Cantor set $K$ in $\R^d$
 described by $(\Sigma_{\mathfrak{B}},g)$ where $g$ is a $\cC^{1+\alpha}$ map with a uniform expansion rate bigger than $\mu^{-1}>1$, we say that $K$ is
\begin{itemize}
\item \emph{affine} if for any $a\in \mathfrak{A}$, $g|_{G(a)}: G(a) \to \R^d$ is an expanding affine map.
    \item \emph{conformal} if $Dg(x)/\|Dg(x)\|_{\it op}\in \operatorname{SO}(d)$ for all $x\in K$,
    \item  \emph{bunched} whenever $g$ is bunched at the Cantor set $K$, that is, there is $N_g\in \N$ such that for all $x \in K$,
    $Dg(x) \in \GL(d,\R)$ and
\begin{equation}\label{eq: definition of bunched cantor sets}
    \|Dg^{N_g}(x)\|_{\textit{op}}\cdot 
   m\left(Dg^{N_g}(x)\right)^{-1} \cdot \mu^{\alpha \cdot N_g } <1.  
\end{equation}
\item \emph{homogeneous} if it is affine and there are $\lambda\in (0,1)$ and $A\in \SL(d,\R)$ such that for any $a\in \mathfrak{A}$, $(g|_{G(a)})^{-1}:g(G(a))\to G(a)$ is the affine map $x\mapsto \lambda\cdot Ax+v_{a}$ for some $v_{a}\in \R^d$. Whenever $(A,\lambda)$ are prefixed, the Cantor set $K$ is called \emph{$(A,\lambda)$-homogeneous}.
\end{itemize}
\end{definition}

Conformal Cantor sets are special cases of bunched Cantor sets. Moreover, the bunching condition is an open condition.
Therefore, small perturbations of conformal Cantor sets are bunched Cantor sets.

Affine Cantor sets also could be described as the unique attractor of a family of affine maps. In particular, given a family of uniformly contracting affine maps $f_1,f_2,\cdots,f_n \in \Aff(d,\R)$ together with an open compact set $U\subset \R^d$ such that 
\begin{itemize}
    \item for any $1\leq i<j\leq n$, $f_i(U)\cap f_j(U)= \emptyset$, and
    \item for any $1\leq i\leq n$, $f_i(U)\subset U$,
\end{itemize}
the unique compact set $K\subset U$ such that
$$K= \bigsqcup_{i=1}^n f_i(K)$$
is then an affine Cantor set with full shift symbolic type $\Sigma$.

\section{Covering conditions}

Covering conditions play a central role in the study of stable intersections of
regular Cantor sets. In particular, the existence of intersections can often
be reduced to verifying certain covering properties. 

In this section we review the notion of covering conditions 
and provide a key lemma for the coverings in linear groups.

\begin{definition}[Covering conditions]\label{def: Covering}
    Let $\cF$ be a family of continuous maps on a metric space $(X,d)$.
    \begin{itemize}
        \item We say that a set $V\subset X$ satisfies the \emph{covering condition} with respect to $\cF$ if
    \begin{equation}\label{eq covering condition}
        \overline{V} \subset \bigcup_{f\in \cF} f^{-1}(V).
    \end{equation}
        \item  We say that a set $V\subset X$ satisfies the {\emph{strong covering}} condition with respect to $\cF$ if there exists $\delta>0$ such that 
    \begin{equation}\label{eq: covering condition delta interior}
         B_{\delta}(V) \subset \bigcup_{f\in \cF} f^{-1}(V_{(\delta)}).
    \end{equation}
    \item We say $V\subset X$ satisfies {$n$-\emph{separable covering condition}} with respect to $\cF$ if one can partition $\cF=\bigsqcup_{i=1}^n \cF_i$ such that for any $1\leq i\leq n$, $V$ satisfies covering condition with respect to the family $\cF_i$.
    \end{itemize}
    \end{definition}

$n$-Separable coverings provide multiple ways to return a given point
$p \in \overline{V}$ into
$V$, a feature that is crucial for
applications. 

Given a finite family $\mathcal{F}=\{f_1,\ldots,  f_k\}$ of continuous maps, we say that the family $\tilde{\cF}$ is $\epsilon$-close to $\cF$ in the $\cC^0$ topology, if $\tilde{\cF}=\{\tilde{f}_1,\ldots,\tilde{f}_k\}$ is such that  $f_i,\tilde{f}_i$ are $\epsilon$-close in the $\cC^0$ topology for $i=1,\ldots,k$. 

For a finite family $\cF$ of continuous maps on a locally compact metric space $(X,d)$, the covering condition implies the strong covering condition for an open relatively compact set $V\subset X$. Moreover, the strong covering condition is stable with respect to the family $\cF$:
 if a set $V\subset X$ satisfies the strong covering condition with respect to a family of continuous maps $\cF$, it remains satisfied for any family $\tilde{\cF}$ in a $\cC^0$ neighborhood of $\cF$ (see \cite[\S 3]{NZ1}).

\subsection{Covering in Linear Groups}
In this section, we study covering properties arising from the group structure
of the Lie groups. By considering left and right actions of group elements, we formulate covering
conditions that will be used later in the renormalization arguments.

Let $G$ be a topological group.
For any pair $(a,b)\in G\times G$, define the continuous map
\[
f_{a,b} : G \to G, \qquad f_{a,b}(g) = a^{-1} g b.
\]
Given a subset $\cF \subset G \times G$ and an open set $\cU \subset G$, we say
that $\cU$ satisfies the \emph{covering condition with respect to $\cF$} if
\begin{equation}\label{eq covering condition for topological groups}
\overline{\cU} \subset \bigcup_{(a,b)\in \cF} f_{a,b}^{-1}(\cU),
\end{equation}
where
\[
f_{a,b}^{-1}(\cU) = \{\, f_{a,b}^{-1}(u) : u \in \cU \,\}.
\]

Since $G$ is locally compact, the strong covering condition and the covering condition are equivalent in this space; hence, we use both terms interchangeably, keeping in mind their equivalence.

In this section, we assume that $G$ is either $\GL(d,\mathbb{R})$ or $\GL(d,\C)$, and that $G$ is endowed with the usual operator norm $\|\cdot\|$. 

Here we prove our key covering lemma, which can be viewed as a generalization of 
\cite[Lemma~3.8]{FNR-AiM}.

\begin{lemma}[Quantitative covering near $\id$]\label{lem: generating a convering in SL with d^2 maps} Let $a\in G$ be an arbitrary element and $\cU_0$ be a neighborhood of $\id$ in $G$. There exists an open set $\cU\subseteq \cU_0$ and a finite set $\cF \subset \cU_0$ such that $\cU$ satisfies strong covering condition with respect to both $\cF a\times \{a\}$ and $\{a\} \times \cF a$ where $\cF a =\{xa: x\in \cF\}$. Moreover, letting $\kappa(a) := \|a\|\,\|a^{-1}\|$, the cardinality of $\cF$ admits
the following bounds.
If $\kappa(a) = 1$, then $|\cF| \le \dim(G) + 1$ and
in the general case
\[
|\cF|
\le \mathrm{e}^{\dim(G) +2}\, \lceil \kappa(a) \rceil^{\dim(G)}.
\]
\end{lemma}
  \begin{proof}[Proof of Lemma \ref{lem: generating a convering in SL with d^2 maps}]
  Let $N:= \dim(G)+1$ and $k:= \lceil \kappa(a)\rceil$.
    Suppose $v_1,\ldots, v_{N}$ are $N$ points in $\mathfrak{G}$ with $|v_i|=1$ such the origin in contained in the interior of their convex hull, which we call $\Delta$. The scale of $\overline{\Delta}$ by the factor $k$ is
    $$k\overline{\Delta}=\left\{\sum_{i=1}^N \ell_i v_i: \sum_{i=1}^N \ell_i=k,\; \ell_1,\cdots,\ell_N\geq 0\right\}.$$
    Denote $\frac{\N^*}{t}:=\{\frac{x}{t}:x\in \N\cup \{0\}\}$ as the set of non-negative rational numbers with denominator $t$. Let
    $$M:=\left\{\sum_{i=1}^N \ell_i v_i: \sum_{i=1}^N \ell_i=k-\dfrac{N}{N+1},\; \ell_i\in
    \dfrac{\N^*}{N+1}\right\}.$$
    We have that $|M|<\max\{{\rm e}^{N+1}\cdot k^{N-1},N\}$ since $|M|=N$ for $k=1$ and for $k>1$,
    $$|M|=\binom{(N+1)k-1}{N-1}= \prod_{j=1}^{N-1} \dfrac{(N+1)k-j}{j} <\dfrac{(N+1)^{N-1}}{(N-1)!}\cdot k^{N-1}<{\rm e}^{N+1}\cdot  k^{N-1}.$$
    We claim that 
    $$k\overline{\Delta}\subseteq \bigcup_{u\in M} {\Delta}+u$$
    where 
    ${\Delta}+u:= \{x+u:x\in{\Delta}\}$. 
    Clearly, given $x=\sum_{i=1}^N \ell_i v_i\in k\overline{\Delta}$, one can find $\tilde{\ell}_i\in \frac{\N^*}{N+1}$ such that $\tilde{\ell}_i\leq \ell_i$ and $\tilde{x}:=\sum \tilde{\ell}_iv_i\in M$. Therefore,
\[
x - \tilde{x}
= \sum_{i=1}^N (\beta_i + \dfrac{1}{N+1} r_i)\, v_i \in {\Delta},
\]
where $\beta_i := \ell_i - \tilde{\ell}_i \geq 0$
 and $r_1,\dots,r_N > 0$ satisfy
\[
\sum_{i=1}^N r_i = 1
\quad \text{and} \quad
\sum_{i=1}^N r_i v_i = 0,
\]
which is possible since $0 \in \Delta$. Since $k\overline{\Delta}$ is compact and $\Delta+u$'s are all open, one can find sufficiently small $c>0$ (in particular $c=O((N+1)^{-1})$) such that $k\overline{\Delta}\subseteq \bigcup_{v\in M}(\Delta+ u)_{(c)}$ (recall that $X_{(\delta)}$ is the $\delta$-interior of $X$). As the whole construction is invariant under homothety, for any  $r>0$,
    \begin{equation}\label{eq:Linear_Convering_of_Simlex_Delta}
    kr\overline{\Delta}=k\overline{r\Delta}\subseteq \bigcup_{u\in M} (r\Delta+r u)_{(cr)}.
    \end{equation}
    
    For any $u\in M$, and any sufficiently small  $r>0$, we define $f^{(r)}_u,g^{(r)}_u:r\Delta\to \mathfrak{G}$ as the following 
    $$f_u^{(r)}(x)=\matrixexp^{-1}\big(\matrixexp(ru)\matrixexp(x)\big),\;\; g_u^{(r)}(x)=\matrixexp^{-1}\big(\matrixexp(x)\matrixexp(ru)\big)$$
    By the Baker-Campbell-Hausdorff formula for the Lie groups (see for instance \cite{Rossmann}), there exists a constant $C=C_{\mathfrak{G}}>0$ such that for a given $u,x\in \mathfrak{G}$ with $\|ru\|,\|x\|<1$, 
    \[\|\matrixexp^{-1}\big(\matrixexp(ru)\matrixexp(x)\big)-(x+ru)\|\leq C\cdot \|x\|\cdot \|ru\|,\]
    \[
    \|\matrixexp^{-1}\big(\matrixexp(x)\matrixexp(ru)\big)-(x+ru)\|\leq C\cdot \|x\|\cdot \|ru\|.
    \]
    Thus, whenever
\[
r < \min_{u \in M} \left\{ \frac{c}{C\,\|u\|} \right\},
\]
we have, for every $u \in M$,
\[
\sup_{x \in \overline{r\Delta}}
\bigl\| f_u^{(r)}(x) - (x + r u) \bigr\|
\leq \sup_{x \in \overline{r\Delta}} C\cdot \|x\|\cdot \|ru\|
< c r.
\]
    Consequently, 
    $(r\Delta+ru)_{(cr)}\subseteq f_u^{(r)}(r\Delta)$ for any $u\in M$. Therefore, (\ref{eq:Linear_Convering_of_Simlex_Delta}) implies that
    $$ k\overline{r\Delta}\subseteq \bigcup\limits_{u\in M} f^{(r)}_u(r\Delta).$$
    Similarly, one can deduce that
    $$ k\overline{r\Delta}\subseteq \bigcup\limits_{u\in M} g^{(r)}_u(r\Delta).$$
    On the other hand, given any $x\in r\Delta \subset\mathfrak{G}$, $$a\matrixexp(x)a^{-1}=\matrixexp(axa^{-1}),\; \|axa^{-1}\|\leq \kappa(a)\|x\|.$$ 
    Therefore,
    $$\matrixexp^{-1}(a\matrixexp(r\Delta)a^{-1})= 
    a\overline{r\Delta}a^{-1}\subseteq k\overline{r\Delta}\subseteq \bigcup\limits_{u\in M} f^{(r)}_u(r\Delta). $$
    So, as the exponential map is a diffeomorphism on $f^{(r)}_u(r\Delta)$ and on $k\overline{r\Delta}$, 
    \begin{equation}\label{eq:covering_equation}
    a\matrixexp(r\Delta)a^{-1}\subseteq 
    \bigcup\limits_{u\in M}\Big(\matrixexp\big(f_u^{(r)}(r\Delta)\big)\Big) = 
    \bigcup\limits_{u\in M} \matrixexp(ru)\matrixexp(r\Delta) .
    \end{equation}
    Similarly we have 
    $$a\matrixexp(r\Delta)a^{-1}\subseteq \bigcup\limits_{u\in M} \matrixexp(r\Delta) \matrixexp(ru). $$
    Thus, the conclusion of Lemma \ref{lem: generating a convering in SL with d^2 maps}  follows from (\ref{eq:covering_equation}) by taking
    $$\cU:= \matrixexp(r\Delta),\; \cF:=\{\matrixexp(-ru): u\in M\}$$ and $r$ small enough such that $\cU,\cF\subset \cU_0$ since $|\cF|=|M|<{\rm e}^{N+1}\cdot k^{N-1}$. 
    \end{proof}

It follows from the proof of Lemma~\ref{lem: generating a convering in SL with d^2 maps}
that the following stronger conclusion even holds.

\begin{proposition}\label{prop: stronger quantitative covering}
   Let $\kappa > 1$ and let $\cU_0$ be a neighborhood of\/ $\id$ in $G$.
Then there exist a neighborhood $\cU \subset \cU_0$ of $\id$, a finite set
$\cF \subset \cU_0$, with 
\[
|\cF|
\le \mathrm{e}^{\dim(G)+2}\, \lceil \kappa \rceil^{\dim(G)}
\]
and a region
\(
W := \bigl\{ a \in G : \|a\|\,\|a^{-1}\| \le \kappa \bigr\}
\)
such that for every $a \in W$
\[
a\,\overline{\cU}\,a^{-1}
\subseteq \left(\bigcup_{f \in \cF} f^{-1}\cU\right) \cap \left(\bigcup_{f \in \cF} \cU f^{-1}\right) .
\]
\end{proposition}

\section{Proof of main theorems}
The aim of this section is to prove Theorems \ref{thm: main1} and \ref{thm: main2}. We provide explicit examples of pairs of stably intersecting Cantor sets by showing that their corresponding renormalization operators satisfy the covering criterion.

Here we recall the covering criterion for a pair of affine Cantor sets $K,K'$, which is a consequence of 
\cite[Theorem~6.6]{NZ1}.

\begin{theorem}[Covering criterion for a pair of affine Cantor sets]\label{thm: covering for affines}
   Let $K,K'\subset \R^d$ be a pair of bunched affine Cantor sets generated with affine maps $\{f_i\}_{i=1}^k,\{f'_{j'}\}_{j'=1}^{k'}$ respectively.
   \begin{itemize}
       \item The associated renormalization operators of the pair $(K,K')$ acting on the space $\Aff(d,\R)$ are given by
   $$R_{i,j'}: X\mapsto f_i^{-1} \circ X \circ f'_{j'}:\;\; X\in \Aff(d,\R).$$
       \item  Given  $f\in \Aff(d,\R)$, the pair $(K,f(K'))$ have stable intersection if there exists an open compact set 
   $\cW \subset \Aff_{\id}(d,\R)$ containing $f$ such that $\cW$ satisfies covering condition with respect to the family of renormalization operators $\{R_{i,j'}\}$ of the pair $(K,K')$.
   \end{itemize}

\end{theorem}

\subsection{Proof of Theorem \ref{thm: main1}}

\begin{proof}[Proof of Theorem \ref{thm: main1}] 

\textbf{Step 1.}
The first step is the content of Lemma \ref{lem: homogeneous cantors on [0,1]}, which is proved in the next sections. In Lemma \ref{lem: homogeneous cantors on [0,1]}, we construct a pair of thin homogeneous Cantor sets $(K_1, K_1')$, both sharing the same contraction ratio $\ell$, satisfying
\[
\dimHD(K_1) \in (\gamma,\, \gamma+\varepsilon), \quad
\dimHD(K_1') \in (1-\gamma,\, 1-\gamma+\varepsilon),
\]
where $\varepsilon > 0$ and $\gamma \in (0, 1)$ are prescribed real numbers.
The construction is arranged so that we can find a closed segment 
\[
W_1 := \{a\} \times I \subset \Aff(1,\mathbb{R}) \cong\mathbb{R}^2
\]
which, viewed as an one-dimensional object, satisfies $c$-separable covering condition 
(for some specified integer $c \in \mathbb{N}$) 
with respect to the renormalization operators of the pair $(K_1, K_1')$. 
In particular, we require covering 
\begin{equation}\label{eq: covering property of W1}
    \{a\} \times I 
   \subset 
   \bigcup R_{i,j'}^{-1}\big(\{a\} \times I_{(\delta)}\big),
\end{equation}
holds, for some positive $\delta>0$, where the union is taken over all pairs of words $(\ua,\ua')$ of equal
length $2$, and that this covering is $c$-separable.
Note that, since the contraction ratios of the generators of $K_1$ and $K_1'$ 
coincide, each renormalization operator $R_{i,j'}$ preserves the scale coordinate in the 
space $\Aff(1,\mathbb{R}) \cong \mathbb{R}^2$. Therefore, the covering condition \eqref{eq: covering property of W1} is stable with respect to the renormalization operators $R_{i,j'}$ of the pair $(K_1,K_1')$, provided that the generators of $K_1$ and $K_1'$ are allowed to vary, while all having the same contraction ratio.

\textbf{Step 2.}
In the second step, we consider the product Cantor sets
\[
K_d := K_1^d, \qquad K_d' := K_1'^{\,d},
\]
and define
\[
W_d := \{a\} \times I^d 
      = \{(a,t_1,t_2,\dots,t_d) : (a,t_i)\in \{a\}\times I\}.
\]
Since $W_1$ satisfies the covering condition by a $c$-separable cover with
respect to the renormalization operators of the pair $(K_1,K_1')$, it follows
that $W_d$ satisfies the covering condition by a $c^d$-separable cover with
respect to the renormalization operators associated to $(K_d,K_d')$.

\begin{figure}
    \centering
    \includegraphics[width=0.5\linewidth]{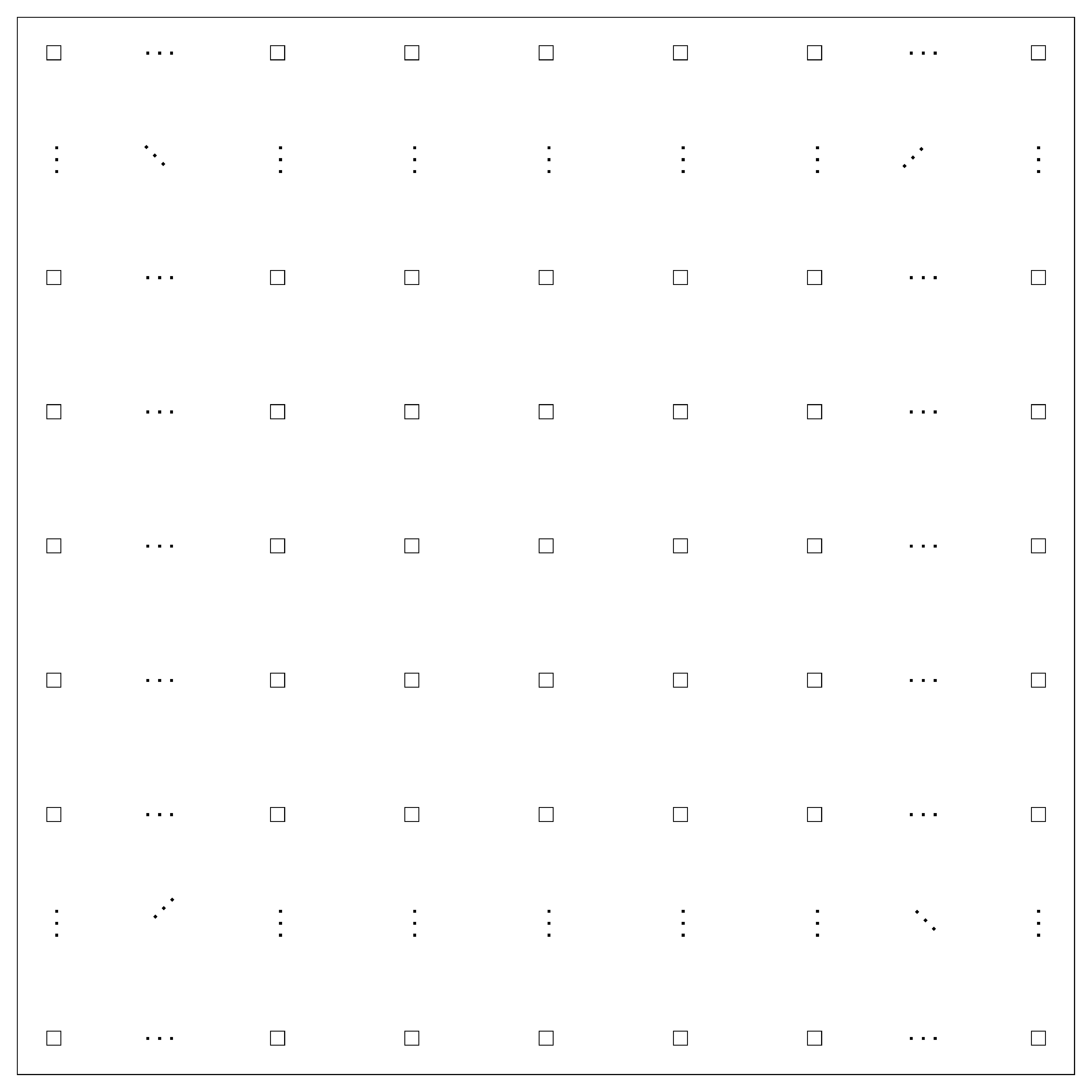}
    \caption{Cantor set $K_d$ in dimension $d=2$, first step approximation.}
    \label{fig:K_d}
\end{figure}

By homogeneity of $K_1$ and $K_1'$, all renormalization operators of the pair
$(K_d,K_d')$ belong to $\Aff_{\id}(d,\R)$. Moreover, since the generators of
$K_1$ and $K_1'$ have identical contraction ratios, the generators of $K_d$ and
$K_d'$ share the same linear part. As a consequence, for any pair of indices $(i,j')$, the corresponding renormalization operator
$R_{i,j'}$ preserves the affine subspace
\(
\{a\}\times\R^d .
 \)
 Restricted to this invariant subspace, the action of the renormalization
operators is well defined, and $W_d$ satisfies a $c^d$-separable covering
property, which is stable under perturbations of the generators, provided that $\{a\}\times \R^d$ remains invariant, similar to \cite[Lemma~7.5]{NZ1}.

\textbf{Step 3.}
We apply Lemma~\ref{lem: generating a convering in SL with d^2 maps} to
$G=\GL(d,\R)$ and $a=\id$. Then, we obtain an open set
\(
\cU \subseteq B_{\delta}(\id) \subset \GL(d,\R)
\)
together with matrices
\(
A_1,\dots,A_{d^2+1} \in B_{\delta}(\id)
\)
such that
\begin{equation}\label{eq:covering in GL}
\overline{\cU} \subset \bigcup_{l=1}^{d^2+1} \cU A_l^{-1} .
\end{equation}

We then define a neighborhood $\cW \subset \Aff(d,\R)$ of $W_d$ by
\[
\cW :=
\bigl\{\, x \mapsto Ax+v : (a,v)\in W_d,\; A\in a\cU \,\bigr\}.
\]

Since $W_d$ satisfies a $c^d$-separable covering property with respect to the
renormalization operators of the pair $(K_d,K'_d)$, this family can be
partitioned into subfamilies
\(
\cH_1,\dots,\cH_{c^d}
\)
such that $W_d$ satisfies the covering condition (inside the invariant subspace
$\{a\}\times\R^d$) with respect to each $\cH_j$.

We then, for a sufficiently small $\delta>0$ which will be determined later, perturb the generators of $K'_d$ so that, for each
$1\le l\le d^2+1$, the action of renormalization operators $R_{i,j'}\in \cH_l$ on the linear parts of the space of affine maps on $\R^d$ coincide and be equal to
$$[X\mapsto a^{-1} \cdot X \circ (a\cdot A_l)] = [X\mapsto X A_l^{-1}],\;\; X\in \gldr.$$
We denote the resulting
Cantor set by $K'$ and set $K:=K_d$.
We claim that for sufficiently small $\delta>0$, $\cW$ satisfies the covering criterion with respect to the
renormalization operators of the pair $(K,K')$. The proof of this claim is similar to the last step of the proof of 
\cite[Theorem~7.7]{NZ1}. Indeed, by \eqref{eq:covering in GL} and that $W_d$ satisfies $c^d$-separable covering for any $p\in \overline{\cW}$ there exists a renormalization operator which maps $p$ into $\cW$.

\begin{figure}
    \centering
    \includegraphics[width=0.9\linewidth]{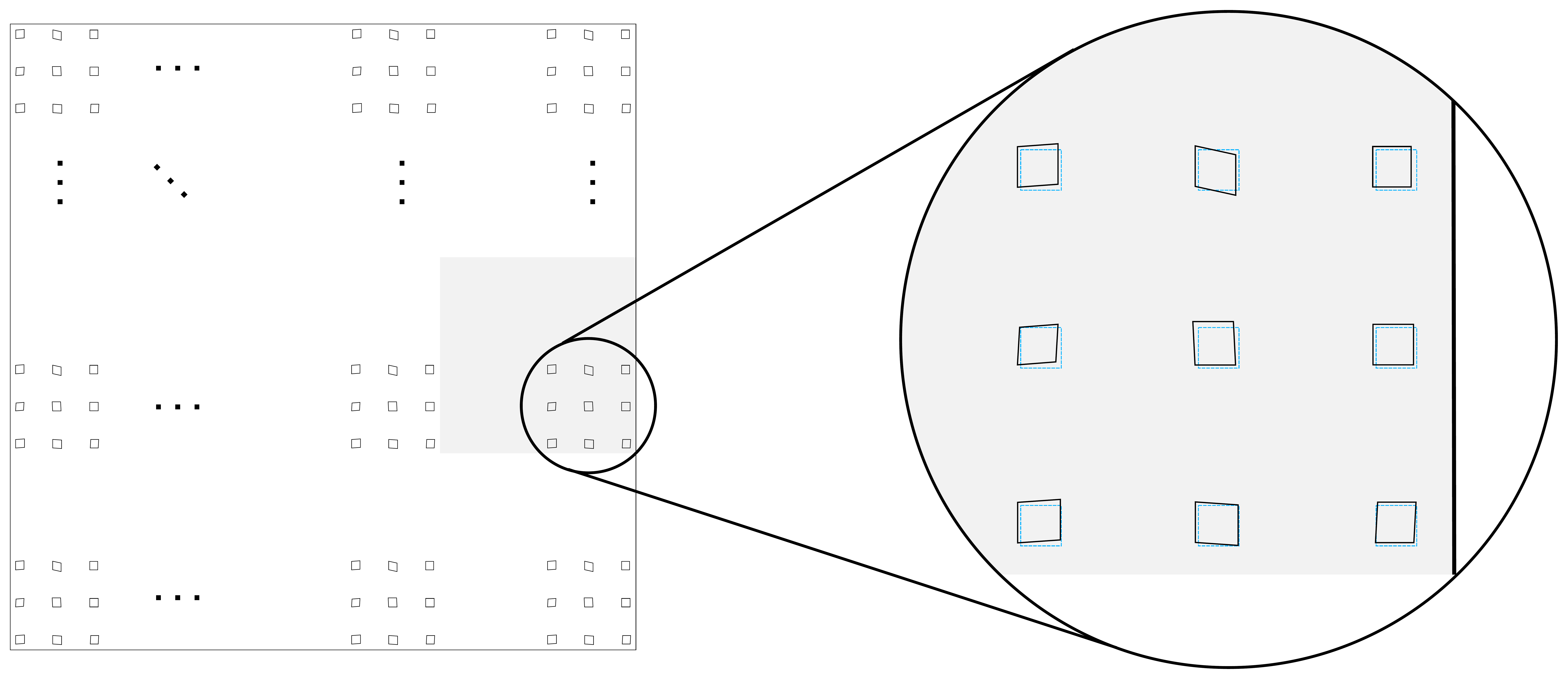}
    \caption{The first step approximation of Cantor set $K'$ on the left, and its zoomed building blocks on the right. The blue rectangles show $K'_d$ in dimension $d=2$.}
    \label{fig:K'-nearlyhomothetical}
\end{figure}


\textbf{Step 4.}
Finally, for the given $p, q\in (0, d)$ with $p+q>d$, let $\delta \in(0, p+q-d)$ and $\epsilon>0$. 
Then, with the construction described above, for $\gamma := (p-\delta)/d$, we obtain a pair of  Cantor sets 
$(K, K') \subset \mathbb{R}^d$ with stable intersection such that 
\[
\dim_{\Aff}(K) \in (p-\delta,\, p - \delta +\epsilon)
\;\;\text{and}\;\;
\dim_{\Aff}(K') \in (d - p+\delta,\, d - p + \delta+\epsilon).
\]
By adding auxiliary affine generators to \(K\) and choosing small and suitable \(\delta ,\varepsilon>0\), the intermediate value theorem allows us to arrange that \(\dim_{\mathrm{Aff}}(K)=p\), using the continuity of the affinity dimension with respect to the affine generators \cite{FengShmerkin}. Although \(\dim_{\mathrm{Aff}}(K)\) need not coincide with \(\dimUB(K)\) and  \(\dim_{\mathrm{H}}(K)\) in general, Falconer’s theorem \cite{FalconerAff} ensures that, after a small perturbation of the translation parts of the generators, equality holds. Moreover, the stable intersection property is preserved under this perturbation, since the covering criterion is open. The same idea of auxiliary generators also helps to construct $K'$ such that one has $\dimUB(K')=\dim_{\Aff}(K')=\dimHD(K')=q>d-q+\delta$.
\end{proof}

\subsection{Proof of Theorem \ref{thm: main2}}

\begin{theorem}\label{thm: non QC Cantors}
Let $d\in \N$, $p,q\in (0,d)$ such that $p+q>d$ and $\varepsilon>0$ be real numbers. There exists $\lambda\in (0,1)$ and a neighborhood $\cU$ of $\id$ in $\SL(d,\R)$ such that for any matrix $A\in \cU$ there exists $\cC^{1+\alpha}$ stably intersecting Cantor sets $K,K'\subset \R^d$ such that
\begin{itemize}
     \item $\dimHD(K)=p$ and 
    $\dimHD(K')=q$.
    \item $K\subset \R^d$ is $(A,\lambda)$-homogeneous Cantor set and $K'\subset \R^d$ is an affine Cantor set which could be arbitrarily close to an $(A,\lambda)$-homogeneous Cantor set.
\end{itemize}
The construction can be adapted to the complex setting as well.
\end{theorem}

\begin{proof}[Proof of Theorem \ref{thm: main2}] Let $K,K'\subset \R^d$ be as in Theorem \ref{thm: non QC Cantors}. Given $n\geq 1$, one can consider the $n$-th iteration of the generators of $K$ which implies that $K$ is also an $(A^n,\lambda^n)$-homogeneous Cantor set. Therefore, by choosing $K'$ close enough to a $(A,\lambda)$-homogeneous Cantor set, $K'$ also can be considered as arbitrarily close to an $(A^n,\lambda^n)$-homogeneous Cantor set. 
This yields Theorem~\ref{thm: main2} for every \(L \in \mathrm{SL}(d,\mathbb{R})\) with non-negative real eigenvalues. Indeed, since diagonalizable (over $\C$) matrices are dense in \(\mathcal U\), it follows that
\[
\mathcal U^* := \{ A^n : A \in \mathcal U,\; n \in \mathbb N \}
\]
is dense in the set of matrices \(L \in \mathrm{SL}(d,\mathbb{R})\) whose real eigenvalues are non-negative.

The case where \(L \in \mathrm{SL}(d,\mathbb{R})\) has \(m\ge 1\) negative eigenvalues is treated analogously, using a corresponding version of Theorem~\ref{thm: non QC Cantors} in which \(\mathrm{Id}\) is replaced by the diagonal matrix \(D \in \mathrm{SL}(d,\mathbb{R})\) whose eigenvalues are \(\pm 1\), with exactly \(m\) entries equal to \(-1\). 
Indeed, the sequence \(\{B^{1/n}\}_{n=2k+1}\) converges to this diagonal matrix when $B$ is diagonal with $m$ negative real eigenvalues. In addition, in this setting, one also requires an analogous version of Proposition~\ref{prop: stronger quantitative covering} adapted to \(D\), as will be shown later.
 \end{proof}

\begin{proof}[Proof of Theorem~\ref{thm: non QC Cantors}]
The proof follows the same steps as in Theorem~\ref{thm: main1}.

    \textbf{Step 1.}
    Let $c \geq 2$ be an integer, and let $K_1, K'_1 \subset [0,1]$ be $\ell$-homogeneous Cantor sets constructed in Lemma~\ref{lem: homogeneous cantors on [0,1]} such that
    \begin{equation}\label{eq: HD, frac d}
    \dimHD(K_1) \in \left(\frac{p}{d} + \frac{\varepsilon}{3d}, \frac{p}{d} + \frac{2\varepsilon}{3d}\right), \quad
    \dimHD(K'_1) \in \left(1 - \frac{p}{d} + \frac{\varepsilon}{3d}, 1 - \frac{p}{d} + \frac{2\varepsilon}{3d}\right).
    \end{equation}
    Then $K_d := K_1^d$ and $K'_d := (K'_1)^d$ are $(\id, \lambda)$-homogeneous Cantor sets generated, respectively, by the families $\{f_i\}_{i=1}^{n^d}$ and $\{f'_{j'}\}_{j'=1}^{n'^d}$ of affine maps
    \[
    f_i : x \mapsto \lambda x + v_i, \quad f'_{j'} : x \mapsto \lambda x + v'_{j'},
    \]
    where $\lambda := \ell^d$. By Lemma~\ref{lem: homogeneous cantors on [0,1]}, there exist $a \in \mathbb{R}$ and an interval $I \subset \mathbb{R}$ such that $W_d := \{a\} \times I^d \subset \{a\} \times \mathbb{R}^d$ satisfies the $c^d$-separable covering condition with respect to the renormalization operators $R_{i,j'}$ of the pair $(K_d, K'_d)$ on $\Aff_{\id}(d, \mathbb{R})$, given by
    \begin{equation*}\label{eq: action of R_i,j'}
    R_{i,j'} : [x \mapsto kx + t] \mapsto [x \mapsto kx + \lambda^{-1}(t + k v'_{j'} - v_i)].
    \end{equation*}
 
    \textbf{Step 2.} 
    Given $A \in \mathrm{SL}(d, \mathbb{R})$, let $K_A, K'_{A}$ be $(\id, \lambda)$-homogeneous Cantor sets generated by the families $\{\hat{f}_i\}_{i=1}^{n^d}, \{\hat{f'}_{j'}\}_{j'=1}^{n'^d}$, where for each $1 \leq i \leq n^d$ and $1 \leq j' \leq n'^d$,
$$ \hat{f}_i : x \mapsto \lambda A x + v_i, \quad \hat{f'}_{j'} : x \mapsto \lambda A x + v'_{j'}. $$
The renormalization operators $\hat{R}_{i,j'}$ of the pair $(K_A, K'_A)$ on $\Aff_{\id}(d, \mathbb{R})$ are given by
$$ \hat{R}_{i,j'} : [x \mapsto kx + t] \mapsto [x \mapsto kx + \lambda^{-1} A^{-1}(t + k v'_{j'} - v_i)]. $$
By the stability of strong covering, there exists a neighborhood $\hat{\cU}$ of $\id$ in $\gldr$ such that for any $A \in \hat{\cU} \cap \mathrm{SL}(d, \mathbb{R})$, $W_d$ satisfies $c^d$-separable covering with respect to the action of renormalization operators of the pair $(K_A, K'_A)$.

\begin{figure}
    \centering
    \includegraphics[width=\linewidth]{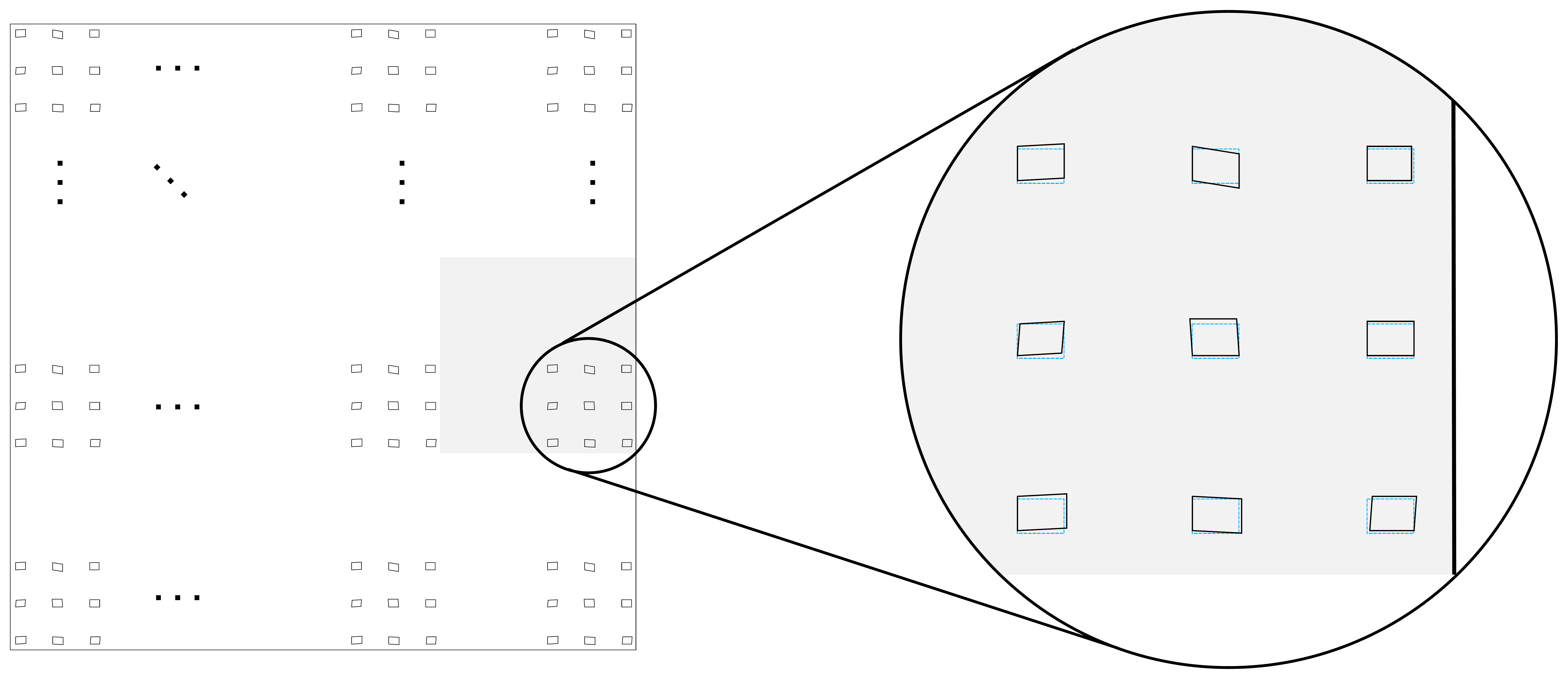}
    \caption{The first step approximation of Cantor set $K'$ on the left, and its zoomed building blocks on the right. The blue rectangles show $K'_A$ in dimension $d=2$, and for some diagonal matrix $A$.}
    \label{fig:K'_A}
\end{figure}

    \textbf{Step 3.} 
    Let $\kappa > 1$ be a fixed real number, and let $c \in \mathbb{N}$ be such that
\[
c^d > e^{d^2 + 2} \lceil \kappa \rceil^{d^2 }.
\]
By Proposition~\ref{prop: stronger quantitative covering}, for any sufficiently small $\delta > 0$ (to be determined later), there exists a neighborhood $\cU_0 \subseteq B_\delta(\id)$ of $\id$ in $\mathrm{GL}(d, \mathbb{R})$ and a finite family
\[
\cF := \big\{ [x \mapsto F_1 x], [x \mapsto F_2 x], \dots, [x \mapsto F_l x] \big\} \subset B_\delta(\id),
\]
where $F_i \in \mathrm{GL}(d, \mathbb{R})$ and $l = |\cF| < c^d$, such that for any $A \in \mathrm{GL}(d, \mathbb{R})$ with $\|A\| \cdot \|A^{-1}\| \leq \kappa$,
\begin{equation}\label{eq: covering dominated}
\overline{\cU_0} \subseteq \bigcup_{f \in \cF} A \cU_0 (A f)^{-1}.
\end{equation}
Let
\[
\cU := \hat{\cU} \cap \left\{ A \in \mathrm{SL}(d, \mathbb{R}) : \|A\| \cdot \|A^{-1}\| <\kappa \right\}
\]
be a neighborhood of $\id$ in $\mathrm{SL}(d, \mathbb{R})$.

Fix $A \in \cU$, and define a neighborhood $\cW \subset \Aff(d, \mathbb{R})$ of $W_d$ by
\[
\cW := \left\{ x \mapsto Bx + v : (B, v) \in \lambda \cU_0 \times I^d \right\}.
\]
Since $W_d$ satisfies a $c^d$-separable covering property with respect to the renormalization operators of the pair $(K_A, K'_A)$, this family can be partitioned into subfamilies $\cH_1, \dots, \cH_{c^d}$ such that $W_d$ satisfies the covering condition (within the invariant subspace $\{a\} \times \mathbb{R}^d$) with respect to each $\cH_j$.

For each $1 \leq j \leq l$, perturb the generators $[x \mapsto \lambda A x + u]$ of $K'_A$ corresponding to $\cH_j$ to $[x \mapsto (\lambda A) \circ F_j x + u]$, so that the action of renormalization operators $R_{i, j'} \in \cH_j$ on the linear parts of $\Aff(d, \mathbb{R})$ becomes
\[
[X \mapsto \lambda^{-1} A^{-1} \circ X \circ (\lambda A) \circ F_j] = [X \mapsto A^{-1} X A \circ F_j], \quad X \in \mathrm{GL}(d, \mathbb{R}).
\]
Let $K := K_A$ and $K'$ be the resulting Cantor set. The conclusion that $\cW$ satisfies the covering criterion with respect to the renormalization operators of $(K, K')$ for sufficiently small $\delta > 0$ follows exactly as we did in the proof of Theorem \ref{thm: main1}, using the covering condition \eqref{eq: covering dominated} and the stability of the covering of $I^d \subset \mathbb{R}^d$ under perturbations of the renormalization operators.

\textbf{Step 4.} By Falconer's theorem \cite{FalconerAff}, for almost every choice of vectors $v_i, v'_{j'} \in \mathbb{R}^d$, the upper box dimension and Hausdorff dimension coincide with the affinity dimension. 
Since stable intersection is an open property with respect to the Cantor sets $K$ and $K'$, we can perturb the vectors $v_i, v'_{j'} \in \mathbb{R}^d$ so that
\[
\dimHD(K) = \dimUB(K)= \dimAff(K), \quad \dimHD(K') =  \dimUB(K')=\dimAff(K').
\]
Moreover, since $K_d$ and $K'_d$ are self-similar Cantor sets, by \eqref{eq: HD, frac d}, we have
\[
\dimUB(K_d)=\dim_{\rm{Aff}}(K_d) = \dim_{\rm{H}}(K_d) = d \cdot \dim_{\rm{H}}(K_1) \in \left(p + \frac{\varepsilon}{3}, p + \frac{2\varepsilon}{3}\right),
\]
\[
\dimUB(K'_d)=\dim_{\rm{Aff}}(K'_d) = \dim_{\rm{H}}(K'_d) = d \cdot \dim_{\rm{H}}(K'_1) \in \left(d - p + \frac{\varepsilon}{3}, d - p + \frac{2\varepsilon}{3}\right).
\]
By \cite{FengShmerkin}, the affinity dimension is continuous with respect to the $\mathrm{GL}(d, \mathbb{R})$-parts of the generators. By choosing $\kappa > 1$ sufficiently close to $1$, the neighborhood $\cU \subset \mathrm{SL}(d, \mathbb{R})$ is close to $\mathrm{id}$, and thus
\[
|\dim_{\rm{Aff}}(K_A) - \dim_{\rm{Aff}}(K_d)| < \frac{\varepsilon}{3}, \quad |\dim_{\rm{Aff}}(K'_A) - \dim_{\rm{Aff}}(K'_d)| < \frac{\varepsilon}{4}.
\]
Using the continuity of the affinity dimension again, for sufficiently small $\delta > 0$,
\[
|\dim_{\rm{Aff}}(K') - \dim_{\rm{Aff}}(K'_A)| < \frac{\varepsilon}{12}.
\]
Therefore,
\[
\dim_{\Aff}(K) \in (p, p + \varepsilon), \quad \dim_{\Aff}(K') \in (d - p, d - p + \varepsilon),
\]
since $K = K_A$. This implies that, within these constructions, the set of attainable values of \(\dim_{\mathrm{Aff}}(K)\) is dense in \((0,\tfrac{d}{2})\). Using again the continuity of the affinity dimension and the openness of the stable intersection, we conclude that for any \(p \in (0,\tfrac{d}{2}]\) one can construct a Cantor set \(K\) with \(\dim_{\mathrm{Aff}}(K)=p\). For \(p>\tfrac{d}{2}\), the auxiliary generator technique discussed in the final part of the proof of Theorem~\ref{thm: main1} applies. Similarly, we may construct $K'$ such that $\dim_{\Aff}(K')=q$ for $q>d-p$. The equalities 
$$\dimUB(K)=\dimHD(K)=\dim_{\Aff}(K)=p,\;\; \dimUB(K')=\dimHD(K')=\dim_{\Aff}(K')=q$$
follow by Falconer's Theorem after a perturbation of $K,K'$.
\end{proof}
\begin{remark}
While the proofs of Theorems \ref{thm: main1} and 
\ref{thm: main2} are presented only for the real setting, the complex case can be treated in the same way by applying Proposition \ref{prop: stronger quantitative covering} for the group $G:=\GL(d,\mathbb{C})$.
\end{remark}

\section{Homogeneous Cantor sets on the real line}

We now prove the following lemma on the homogeneous pairs of Cantor sets in $[0,1]$, which constitutes a key 
element
in the proofs of Theorems~\ref{thm: main1} and~\ref{thm: main2}. Owing to the complexity of the underlying construction, the proof is carried out in several stages, which we describe in detail.

\begin{lemma}\label{lem: homogeneous cantors on [0,1]}
Let $\gamma \in (0,1)$, $c\geq 2$ a given integer and $\varepsilon>0$ be an arbitrary positive real number.
There exists a pair of $\ell$-homogeneous Cantor sets $K_1,K_1'\subset [0,1]$ such that
\begin{itemize}
    \item $K_1$ and $K_1'$ are generated by the families of affine maps $\{f_i\}_{i=1}^n$ and $\{f'_j\}_{j=1}^{n'}$ respectively, such that
$$\forall i\in \{1,\dots,n\}: \;f_i(x)= \ell x+ t_i,\;\; \forall j\in \{1,\dots,n'\}: \;f'_j(x)= \ell x+ t'_j,$$
where $\ell\in (0,1)$ is some real number and $n, n'$ are sufficiently large positive integers such that
     $$\dimHD(K_1) \in (\gamma, \gamma+\varepsilon),\;\dimHD(K_1')\in (1-\gamma,1-\gamma+\varepsilon).$$
    \item There exist $a\in \R^+$ and an interval $I\subset \R$ such that the set $\{a\}\times I$ is a recurrent set under the action of renormalization operators of the pair $(K_1,K_1')$. Indeed, $\{a\}\times I$ satisfies the $c$-separable covering condition with respect to the renormalization operators of the pair $(K_1,K_1')$.
\end{itemize}
Moreover, $K_1,K'_1$ can be arbitrarily thin Cantor sets. Indeed, $\tau(K_1),\tau(K'_1)<\varepsilon$ where $\tau(K)$ is the thickness of $K$. 
\end{lemma}

The renormalization operators $\{R_i\}_{i=1}^n$ and $\{R'_j\}_{j=1}^{n'}$ associated to this pair are affine maps acting   (respectively) from left and right on the space $\Aff(1,\R)$ as:
\begin{align}
     \label{eq:renormalization operators:K_1}
R_i^{-1} &: (s,t)\mapsto \left(\ell^{-1}s,\ell^{-1}(t-t_i)\right),\\  \label{eq:renormalization operators:K'_1}
R'_j&: (s,t)\mapsto (\ell s, st'_j+t).    
\end{align}

\begin{figure}[t]
    \centering
       \includegraphics[width=\linewidth]{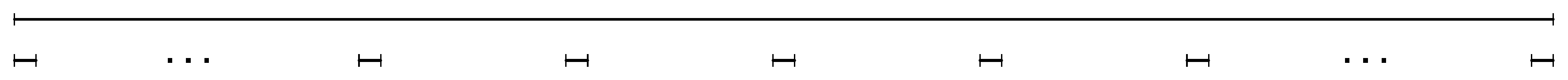}
    \caption{First step approximation of the Cantor set $K_1$.}
    ~\medskip~
    \label{fig:K1-any-dim}
\end{figure}
\begin{figure}
    \centering
    \includegraphics[width=\linewidth]{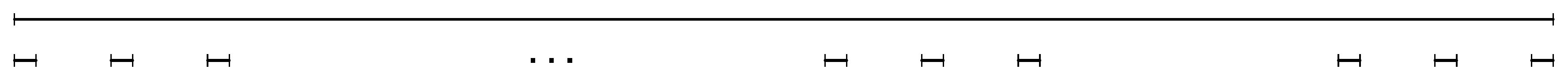}
    \caption{First step approximation of the Cantor set $K_1'$}
    \label{fig:K'1-any-dim}
\end{figure}

It is notable to observe that the operators $\{R'_j\}_{j=1}^{n'}$ are acting like translations on the fibers, and the operators 
$\{R_i^{-1}\}_{i=1}^n$ are homotheties in the following manner:
\begin{align}\label{eq:homomthety and translation renormalization operators}
R_i^{-1}& : (s,\tilde{t}+\dfrac{t_i}{1-\ell})\mapsto \left(\ell^{-1}s,\ell^{-1}\tilde{t} +\dfrac{t_i}{1-\ell}\right).
\end{align}

We claim that the Cantor sets $K_1$ and $K'_1$, constructed using the following data, satisfy the desired properties stated in Lemma~\ref{lem: homogeneous cantors on [0,1]}.

Let
\begin{equation}\label{eq: parameter choosing 1}
\beta := 7,\qquad
\ell := (3c)^{\frac{-1}{2\varepsilon'}},\qquad
n := \bigl\lfloor \ell^{-(\gamma+\varepsilon')} \bigr\rfloor,\qquad
n' := \ell^{-(1-\gamma+\varepsilon')},
\end{equation}
and set
\begin{equation}\label{eq: parameter choosing 2}
|I| := \sqrt{n\ell},\qquad
|J| := (n\ell)^{-1/2},\qquad
a := \ell^{1/2},
\end{equation}
where $\varepsilon' < \min\{\varepsilon,\, \gamma/7\}$ is chosen so that $\ell^{-(1-\gamma+\varepsilon')}$ is an integer divisible by $c$. For $1 \leq i \leq n$, define
\begin{align}\label{eq: parameter choosing 3}
     t_i := z_0 - \dfrac{(i-1)|I|}{n} - \ell w_0,\; i\in \{1,2,\cdots,n\}.
\end{align} 
For $1 \leq k \leq n'/c$ and $1 \leq l \leq c$, define
\begin{equation}\label{eq: parameter choosing 4}
t'_{c(k-1)+l}:=
\dfrac{\ell}{a}\left(z_0-\delta a -w_0 - \dfrac{c-l}{\beta(c-1)}\left(|I|-2\delta a-\dfrac{|J|}{n''}\right) +\dfrac{(k-1)|J|}{n''}\right).
\end{equation}
Then, there exist $w_0$ such that the interval $I := [w_0 - |I|, w_0]$ satisfies the required properties.

In the remainder of this paper, we provide a construction such that the parameters defined above satisfy the desired constraints. At the end, we will discuss how the thickness could be arbitrarily small. 

\subsection*{Steps of construction}

We will choose parameters 
\[
a, n, n', \ell, \delta,\{t_i\}_{i=1}^{n}, \{t'_j\}_{j=1}^{n'},
\]
and intervals 
\[
I = [\,z_0 - |I|,\, z_0\,] \subset \mathbb{R}, 
\; 
J = [\,w_0 - |J|,\, w_0\,] \subset \mathbb{R},
\]
such that the following properties hold:

\begin{itemize}
    \item[(A)] The affine maps $\{x \mapsto \ell x + t_i\}_{i=1}^{n}$ generate a Cantor set 
    $K_1 \subset [0,1]$ with 
    \[
    \dimHD(K_1) \in (\gamma,\, \gamma+\varepsilon).
    \]

    \item[(B)] The affine maps $\{x \mapsto \ell x + t'_j\}_{j=1}^{n'}$ generate a Cantor set 
    $K_1' \subset [0,1]$ with 
    \[
    \dimHD(K_1') \in (1-\gamma,1-\gamma+\varepsilon).
    \]

    \item[(C)] For every point $x \in \{a\} \times I$, there exists a renormalization operator 
    $R_i^{-1}$ associated to $K_1$ such that 
    \[
    R_i^{-1}(x) \in \{\ell^{-1} a\} \times J.
    \]

    \item[(D)] The renormalization operators $\{R'_j\}_{j=1}^{n'}$ associated to $K_1'$ can be 
    partitioned as $\bigsqcup_{l=1}^{c} \mathcal{F}_l$ such that, for every 
    $y \in \{a\} \times I$ and every $l \in \{1,\dots,c\}$, there exists 
    $R' \in \mathcal{F}_l$ with 
    \[
    R'(y) \in \{a\} \times I_{(\delta)}.
    \]
\end{itemize}

The main difficulty of the construction is to satisfy (C) and (D) in a way that remains compatible with (A) and (B).  
Below we give the outline; in the next section we compute the precise constraints needed to ensure (A) and (B).


    \textbf{Performing (C).}  
    By \eqref{eq:homomthety and translation renormalization operators}, the renormalization operators 
    $\{R_i\}_{i=1}^{n}$ associated to $K_1$ are homotheties with center 
    \[
    y_i := \left(0,\, \frac{t_i}{1-\ell}\right)
    \]
     and scale factor $\ell^{-1} > 1$.  
    Denote by $H_i$ the homothety with center $y_i$ and ratio $\ell^{-1}$.  
    We choose the points $y_i$ so that
    \[
    \{a\} \times I 
        = \bigsqcup_{i=1}^{n} 
          H_i^{-1}\!\left(\{\ell^{-1} a\} \times J\right).
    \]
    This implies that the translations $t_1,\dots,t_n$ can be chosen so that 
    $I$ admits a partition (overlaps on end-points of intervals is allowed) $I = \bigsqcup_{i=1}^{n} I_i$ with 
    \[
    |I_1| = \cdots = |I_n|=|I|/n,
    \;\; 
    R_i^{-1}\big(\{a\}\times I_i\big) 
        = \{\ell^{-1} a\} \times J 
        \;\text{for every } i.
    \]

    \textbf{Performing (D).}  
    By \eqref{eq:renormalization operators:K'_1}, each 
    renormalization operator associated to $K_1'$ is a translation from 
    $\{\ell^{-1} a\} \times \mathbb{R}$ to $\{a\} \times \mathbb{R}$.  
    Let $n'' = n'/c$ and partition 
    \[
    J = \bigsqcup_{k=1}^{n''} J_k,
    \;\; |J_1|=\cdots=|J_{n''}|=|J|/n''.
    \]
    We choose
\[
0 \le t'_1 < t'_2 < \cdots < t'_{n'} \le 1-\ell
\]
so that for all $i \in \{1, \dots, n''-1\}$,
$t'_{i+1} - t'_i > \ell$ and
there exist intervals 
\(\widetilde{I}_1, \dots, \widetilde{I}_c \subset I_{(\delta)}\), 
to be chosen later, each of length \(|J|/n''\), satisfying
\[
R'_{\,c(k-1)+l}(J_k) = \widetilde{I}_l,
\qquad 
k = 1, \dots, n'',\; l = 1, \dots, c.
\]
Defining
\(
\mathcal{F}_l := \{ R'_{\,c(k-1)+l} \}_{k=1}^{n''},
\)
property (D) then follows immediately.

We will follow the steps of the construction of 
$K_1$ and $K_1'$, and we compute all the constraints on the parameters involved.

\subsection*{Constraints on $\dimHD$}
We have 
$$\dimHD(K_1)=  \dfrac{\log n}{-\log \ell}, \quad \dimHD(K_1')=  \dfrac{\log n'}{-\log \ell}.$$  
Therefore, the parameters must satisfy
\begin{equation}\label{eq:dimHD conditions}
    \boxed{n\in (\ell^{-\gamma},\ell^{-(\gamma+\varepsilon)} ), \quad
 n'\in \left(\ell^{-(1-\gamma)},\ell^{-((1-\gamma)+\varepsilon)} \right).} 
\end{equation}
At the end, we will choose 
\begin{equation}
n := \left\lfloor \ell^{-(\gamma+\varepsilon')} \right\rfloor, 
\;\; 
n' := \left\lfloor \ell^{-(1-\gamma+\varepsilon')} \right\rfloor, \;\; \varepsilon'\in (0,\varepsilon),
\end{equation}
where $\lfloor x \rfloor$ denotes the integer part of $x$.

\subsection*{Constraints on the construction of $K_1$}

We partition the interval $I$ into $n$ sub intervals
\[
I = \bigsqcup_{i=1}^{n} I_i, \qquad 
I_i = [z_i, z_{i-1}], \quad 
z_i := z_0 - \frac{i |I|}{n}, \quad i = 1, \dots, n.
\]
We then choose the parameters $\{t_i\}_{i=1}^n$ such that:

\begin{itemize}
    \item[(a)] For each $i = 1, \dots, n$, the image of the operator $R_i^{-1}$ restricted to $\{a\} \times I_i = \{a\} \times [z_i, z_{i-1}]$ satisfies
    \[
    R_i^{-1}(\{a\} \times I_i) = \{\ell^{-1} a\} \times J = \{\ell^{-1} a\} \times [w_0 - |J|, w_0].
    \]
    
    \item[(b)] The parameters satisfy
    \[
    1-\ell \ge t_1 > t_2 > \cdots > t_n \ge 0, 
    \qquad t_i - t_{i+1} > \ell \text{ for } i = 1, \dots, n-1.
    \]
    This ensures that the iterated function system 
    $\{x \mapsto \ell x + t_i\}_{i=1}^n$ generates a Cantor set $K_1 \subseteq [0,1]$.
\end{itemize}
\begin{figure}
    \centering
    \includegraphics[width=0.9\linewidth]{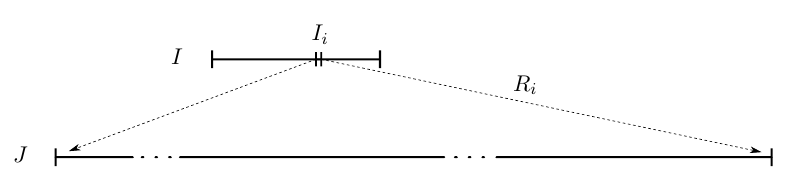}
    \caption{The renormalization operator $R_i$ maps the interval $I_i$ onto the whole interval $J$ via a homothety.}
    \label{fig1}
\end{figure}

\noindent Using the formulas in \eqref{eq:renormalization operators:K_1}-\eqref{eq:renormalization operators:K'_1} we have 
$$R_i^{-1}\left(\{a\}\times [z_{i},z_{i-1}] \right) = \{\ell^{-1}a\}\times [\ell^{-1}(z_{i}-t_i),\ell^{-1}(z_{i-1}-t_i)].$$
Thus, in order for the mapping to behave as required, we need
$$\ell^{-1}(z_{i}-t_i)=w_0-|J|,\;\ell^{-1}(z_{i-1}-t_i)= w_0.$$
Hence, it suffices to require that
\begin{equation}\label{eq:condition on lengths for the R_i operators}
    \boxed{\ell^{-1} \times \dfrac{|I|}{n} = |J|}
\end{equation}
and to choose the parameters $t_i$ as
\begin{align}\label{eq:def of ti}
     t_i := z_0 - \dfrac{(i-1)|I|}{n} - \ell w_0,\; i\in \{1,2,\cdots,n\}.
\end{align} 
Hence, $t_i-t_{i+1}=|I|/\ell$ and  $t_i\in [0,1-\ell]$ if and only if
$$1-\ell \geq z_0-\ell w_0,\;\;z_0 - \dfrac{(n-1)|I|}{n} - \ell w_0\geq 0.$$
Therefore, the constraints in (b) is reduced to
\begin{equation}\label{eq:condition of K is a cantor set or not}
   \boxed{\dfrac{|I|}{n}>\ell,}
\end{equation}
\begin{equation}\label{eq:z0-aw0}
    \boxed{1-\ell \geq z_0-\ell w_0 \geq \dfrac{(n-1)|I|}{n}.}
\end{equation}

\subsection*{Constraints on the construction of $K_1'$}
We partition the interval $J$ into 
$n'' := {n'}/{c}$ sub-intervals 
\[
J = \bigsqcup_{k=1}^{n''} J_k, \qquad 
J_k = [w_k, w_{k-1}], \quad 
w_k := w_0 - \frac{k |J|}{n''}, \quad k = 1, \dots, n''.
\]
 Since the family of operators $\{R'_j\}_{j=1}^{n'}$ acts like translations on the fibers, we define the intervals $\{\tilde{I}_l\}_{l=1}^{c}$ of length $|J|/n''$ by
$$\tilde{I}_l=\left[\tilde{z}_l-\dfrac{|J|}{n''}, \tilde{z}_l\right],\;
\tilde{z}_l := z_0-\delta a - 
\dfrac{c-l}{\beta(c-1)}\left(|I|-2\delta a-\dfrac{|J|}{n''}\right) ,\;l\in \{1,,2,\cdots,c\},$$
where $\beta\geq 2$ is a real number to be defined later.

\begin{figure}
    \centering
    \includegraphics[width=0.9\linewidth]{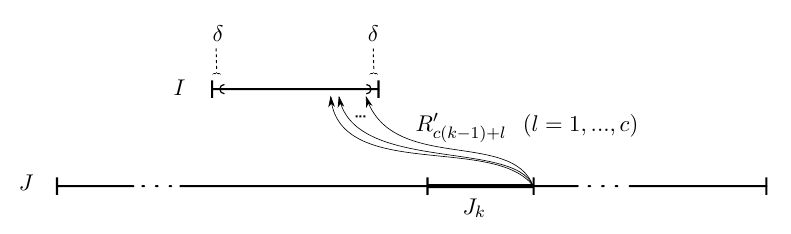}
    \caption{For all $1\leq l\leq c$, the operator $R'_{c(k-1)+l}$ maps the interval $J_k$ onto $\tilde{I}_l\subset I_{(\delta)}$ via a translation.}
    \label{fig2}
\end{figure}

\noindent We then choose the parameters $\{t'_j\}_{j=1}^{n'}$ such that: 
\begin{itemize}
    \item[(a$'$)] For $k = 1, \dots, n''$ and $l = 1, \dots, c$, the image of the operator 
    $R'_{c(k-1)+l}$ restricted to 
    \[
    \{\ell^{-1} a\} \times J_k = \{\ell^{-1} a\} \times [w_k, w_{k-1}]
    \]
    satisfies
    \[
    R'_{c(k-1)+l}(\{\ell^{-1} a\} \times J_k) = \{a\} \times \widetilde{I}_l.
    \]
    \item[(b$'$)] The parameters satisfy
    \[
    1-\ell \ge t'_{n'} > t'_{n'-1} > \cdots > t'_1 \ge 0, 
    \;\; 
    t'_{j+1} - t'_j > \ell \text{ for } j = 1, \dots, n'-1.
    \]
    This ensures that the iterated function system 
    $\{x \mapsto \ell x + t'_j\}_{j=1}^{n'}$ generates a Cantor set 
    $K_1' \subseteq [0,1]$.
\end{itemize}
Using the formulas in \eqref{eq:renormalization operators:K_1}-\eqref{eq:renormalization operators:K'_1} we have 
$$R'_{c(k-1)+l}\left(\{\ell^{-1}a\}\times [w_{k},w_{k-1}] \right) 
= \{a\}\times [w_{k}+\ell^{-1}a\cdot t'_{c(k-1)+l},w_{k-1}+\ell^{-1}a\cdot t'_{c(k-1)+l}].$$
Thus, in order for the mapping to behave as required, we need
$$
w_{k}+\ell^{-1}a\cdot t'_{c(k-1)+l} = \tilde{z}_l-\dfrac{|J|}{n''},\; 
w_{k-1}+\ell^{-1}a\cdot t'_{c(k-1)+l} = \tilde{z}_l.
$$
Hence, it suffices to choose the parameters $t_i$ as
\begin{equation}\label{eq:def of t'j}
t'_{c(k-1)+l}:=
\dfrac{\ell}{a}\left(z_0-\delta a -w_0 - \dfrac{c-l}{\beta(c-1)}\left(|I|-2\delta a-\dfrac{|J|}{n''}\right) +\dfrac{(k-1)|J|}{n''}\right).
\end{equation}
Next, using \eqref{eq:def of t'j}, we rewrite the constraint 
\(t'_{j+1} - t'_j > \ell\) appearing in (b$'$). Indeed,
$$t'_{c(k-1)+l+1}-t'_{c(k-1)+l}=\dfrac{\ell}{a\beta(c-1)}\left(|I|-2\delta a- \dfrac{|J|}{n''}\right)>\ell,\; l\in\{1,\cdots,c-1\},$$
$$t'_{ck+1}-t'_{c(k-1)+c} = \dfrac{\ell}{\beta a}\left(\dfrac{|J|}{n''}\left(\beta-1\right)-\left(|I|-2\delta a\right)
\right)>\ell,\; k\in \{1,\cdots,n''-1\}.$$
Therefore, it suffices to have 
\begin{align}\label{eq:1: constraint on K' become Cantor}
   \boxed{\dfrac{|J|}{a\cdot n''}(\beta-1) -\beta>\dfrac{|I|}{a}-2\delta >\dfrac{|J|}{a\cdot n''}+\beta(c-1).}
\end{align}
Finally, it remains to check the constraint $t'_j\in[0,1-\ell]$ which is included
in (b$'$). Since the sequence $(t'_j)_{j=1}^{n'}$ is strictly increasing, it is enough
to verify that $0\leq t'_1$ and $t'_{n'}\leq 1-\ell$. These are true if and only if
$$ z_0-\delta a-w_0\geq\dfrac{1}{\beta}\left(|I|-2\delta a-\dfrac{|J|}{n''}\right),
$$
$$ 
1-\ell \geq \dfrac{\ell}{a}
\left(z_0-\delta a-w_0+\dfrac{(n''-1)|J|}{n''}\right).
$$
Therefore, the constraints in (b$'$) are reduced to \eqref{eq:1: constraint on K' become Cantor} and 
\begin{equation}\label{eq:2: constraint on K' become Cantor}
   \boxed{\dfrac{a(1-\ell)}{\ell}-\dfrac{(n''-1)|J|}{n''}+\delta a\geq
    z_0-w_0 \geq \delta a\left(1-\dfrac{2}{\beta}\right) + \dfrac{1}{\beta}\left(|I|-\dfrac{|J|}{n''}\right).}
\end{equation}

The verification that the parameter choices given in \eqref{eq: parameter choosing 1},\eqref{eq: parameter choosing 2},\eqref{eq: parameter choosing 3} and \eqref{eq: parameter choosing 4} fulfill all the above conditions given in
\eqref{eq:dimHD conditions},
\eqref{eq:condition on lengths for the R_i operators},
\eqref{eq:condition of K is a cantor set or not},
\eqref{eq:z0-aw0},
\eqref{eq:1: constraint on K' become Cantor},
and
\eqref{eq:2: constraint on K' become Cantor}
is deferred to the subsequent arguments.

\subsection*{Simplification of the Constraints}

In this section, we reduce the previously obtained constraints to a simpler set of sufficient conditions. 
By strengthening certain inequalities and eliminating parameters that can be taken arbitrarily small, we obtain a manageable system that still guarantees the construction.

By \eqref{eq:condition on lengths for the R_i operators}, the length $|J|$ is determined in terms of $|I|$ and $\ell$. Moreover, $\delta$ appears only in \eqref{eq:1: constraint on K' become Cantor} and \eqref{eq:2: constraint on K' become Cantor}. Consequently, the existence of $\delta>0$ as an arbitrarily small positive number is equivalent to replacing the constraints \eqref{eq:1: constraint on K' become Cantor}, \eqref{eq:2: constraint on K' become Cantor} by the strict inequalities in which $\delta$ is set to zero. 
Thus, it suffices to impose the following two constraints, 
in which $\delta$ is omitted, $n''$ is replaced by $n'/c$, and, by \eqref{eq:condition on lengths for the R_i operators}, $|J|$ is replaced with $(n\ell)^{-1}|I|$. The first one is
\begin{equation}\label{eq:1: z0-w0}
\dfrac{a(1-\ell)}{\ell}-\dfrac{1}{\ell n}\left(1-\dfrac{c}{n'}\right) |I|>
    z_0-w_0 > \dfrac{|I|}{\beta}\left(1-\dfrac{c}{\ell n n'}\right)
\end{equation}
and the second one is
$$
    \dfrac{\ell^{-1} c|I|}{a\cdot n n'}(\beta-1) -\beta>\dfrac{|I|}{a} >\dfrac{\ell^{-1} c|I|}{a\cdot n n'}+\beta(c-1),
$$
which is equivalent to
\begin{equation}\label{eq:1: I/a}
    \dfrac{|I|}{a}\left(1-\dfrac{c}{\ell n n'}\right)>\beta(c-1),\qquad
    \dfrac{|I|}{a}\left(\dfrac{ c(\beta-1)}{\ell n n'}-1\right)>\beta.
\end{equation}
Therefore, we must have 
$\ell n n'\in (c,c(\beta-1))$ which is necessary for \eqref{eq:1: I/a} to hold but not enough. We impose the auxiliary constraint
\begin{equation}\label{eq:ann' condition}
    \ell n n' \in [2c,c(\beta-1)/2]
\end{equation}
on $a,n,n',\beta$ which allows us to replace
 \eqref{eq:1: I/a} by 
\begin{equation}\label{eq:final lower bound I/a}
    \dfrac{|I|}{a}>  2\beta(c-1).
\end{equation}
This is a stronger constraint since
$2\beta(c-1)>\beta$ and by \eqref{eq:ann' condition} we have that
$$1-\dfrac{c}{\ell n n'}>\dfrac{1}{2},\qquad
\dfrac{ c(\beta-1)}{\ell n n'}-1>2\beta(c-1).
$$
The only constraints on $z_0,w_0$ are given by \eqref{eq:1: z0-w0} and \eqref{eq:z0-aw0}, which require that
$z_0-w_0$ and $z_0-\ell w_0$ belong to certain intervals. Since the system of equations 
$$z_0-w_0 = x,\qquad z_0-\ell w_0= y$$
has a real solution for generic $x,y$, the existence of $z_0,w_0$ reduces to verifying that the intervals in \eqref{eq:1: z0-w0}, \eqref{eq:z0-aw0} have positive lengths.
Therefore, one can replace the constraints
\eqref{eq:z0-aw0}, \eqref{eq:1: z0-w0} with 
\begin{equation}\label{eq:upper bound I}
    1-\ell > \left(1-\dfrac{1}{n}\right) |I|,
\end{equation}

$$
\dfrac{a(1-\ell)}{\ell}-\dfrac{1}{\ell n}\left(1-\dfrac{c}{n'}\right) |I|>
\dfrac{|I|}{\beta}\left(1-\dfrac{c}{\ell nn'}\right)
$$
which is equivalent to
\begin{equation}\label{eq:I/a upper bound}
   \dfrac{|I|}{a}\left(\left(\dfrac{1}{\beta}+1\right)\left(1-\dfrac{c}{\ell n n'}\right)-1 +\dfrac{1}{\ell n}\right)< \ell^{-1}-1.
\end{equation}
By \eqref{eq:ann' condition} we have
$$
\left(\dfrac{1}{\beta}+1\right)\left(1-\dfrac{c}{\ell n n'}\right)-1\leq \left(\dfrac{1}{\beta}+1\right)\left(1-\dfrac{2}{\beta-1}\right)-1
$$
which is a negative number for $\beta>1$. Therefore, since $\ell^{-1}/2<\ell^{-1}-1$, the inequality \eqref{eq:I/a upper bound} is satisfied provided that 
\begin{equation}\label{eq:final upper bound on I/a}
   \dfrac{|I|}{a}<\dfrac{n}{2}.
\end{equation}
On the other hand, since $n\ell < 1$, the inequality \eqref{eq:upper bound I} is guaranteed by
imposing the additional constraint
\begin{equation}\label{eq:I<1}
    |I| < 1.
\end{equation}

In summary, the full collection of constraints 
\eqref{eq:dimHD conditions}, 
\eqref{eq:condition on lengths for the R_i operators}, 
\eqref{eq:condition of K is a cantor set or not}, 
\eqref{eq:z0-aw0}, 
\eqref{eq:1: constraint on K' become Cantor}, 
\eqref{eq:2: constraint on K' become Cantor}
is automatically satisfied once the simplified system of conditions
\eqref{eq:dimHD conditions}, 
\eqref{eq:condition of K is a cantor set or not}, 
\eqref{eq:ann' condition}, 
\eqref{eq:final lower bound I/a}, 
\eqref{eq:final upper bound on I/a}, 
and \eqref{eq:I<1}
holds. Therefore, the problem of choosing admissible parameters reduces to verifying only these latter inequalities.

\subsection*{The chosen parameters}

 Finally, we choose the parameters $|I|$ and $a$ as
\[
|I| := \sqrt{n \ell}, \qquad a := \sqrt{\ell},
\]
which clearly satisfy \eqref{eq:condition of K is a cantor set or not}, \eqref{eq:final lower bound I/a}, \eqref{eq:final upper bound on I/a} and \eqref{eq:I<1} provided that $n,\beta$ satisfy 
\begin{equation}\label{eq:I deleted}
    \sqrt{n}>2\beta(c-1),\; \sqrt{n}>2.
\end{equation}
Therefore, the existence of such construction for the Cantor sets $K_1,K_1'$ is depended on the bounds on $n,n',\ell,\beta$ obtained in \eqref{eq:dimHD conditions}, \eqref{eq:ann' condition} and \eqref{eq:I deleted}. By \eqref{eq:dimHD conditions}, it suffices to find $\ell,\beta$ and $\varepsilon'\in (0,\varepsilon)$ such that for 
$$n:=\lfloor\ell^{-(\gamma+\varepsilon')}\rfloor, \;\;
n':=\lfloor\ell^{-(1-\gamma+\varepsilon')}\rfloor$$
the relations \eqref{eq:I deleted} and \eqref{eq:ann' condition} are satisfied. Thus, we need to check that
$$\ell^{-2\varepsilon'}\in [2c,c(\beta-1)/2],\;\; \ell^{-(\gamma+\varepsilon')}>4\beta^2(c-1)^2+1$$
which can be achieved by choosing $$\beta:=7,\;\;,\;\ell := (3c)^{\frac{-1}{2\varepsilon'}}$$
where $\varepsilon'>0$ is any real number such that $\varepsilon'< \min\{\varepsilon,\gamma/7\}$. 
\begin{remark}
    Another constraint in the construction of $K_1'$ is that $n'$ must be divisible by $c$.
Recall that
\[
n'=\Big\lfloor (3c)^{\frac{1-\gamma+\varepsilon'}{2\varepsilon'}} \Big\rfloor.
\]
Since we are free to choose $\varepsilon' \in (0,\min\{\varepsilon, \gamma/7\})$, we may
decrease $\varepsilon'$ slightly until the resulting value of $n'$ becomes divisible by
$c$. 
\end{remark}

\subsection*{Thickness}
Given a homogeneous Cantor set $K$ the thickness of $K$ is given by
$$\tau(K)=\min_{i}{\dfrac{\ell}{|G_i|}}$$
where $\ell$ is the contraction ratio of the generators of $K$ and $G_i$ is any gap that appears in the first iteration of the construction of $K$.

We show that the thicknesses $\tau(K_1)$ and $\tau(K'_1)$ can be made arbitrarily small; in particular, the Cantor sets
$K_1$ and $K'_1$ are very thin.

By construction, $K_1$ is a homogeneous Cantor set with contraction ratio $\ell>0$ and equal-sized gaps of length
\[
\frac{1-n\ell}{n-1}.
\]
This implies that
\[
\tau(K_1)
= \frac{(n-1)\ell}{1-n\ell}
= n\ell \cdot \frac{n-1}{n(1-n\ell)}
< 2n\ell
\]
which tends to zero as $\varepsilon'\to 0$.

Unlike $K_1$, the gaps of $K'_1$ are not uniformly ordered, since the construction is designed to produce a
$c$-separable covering.
Nevertheless, this feature can be exploited to ensure that $K'_1$ is also a thin Cantor set while retaining the desired
$c$-separable covering property.

Let $N\in\mathbb{N}$ be arbitrary and consider the construction of the pair $(K_1,\tilde K'_1)$ with $c$ replaced by
$c' := cN$.
Then $\tilde K'_1$ is a homogeneous Cantor set generated by the family
\(
\{x \mapsto \ell x + t'_j\}_{j=1}^{n'}
\)
where
\[
0 \le t'_1 < t'_2 < \cdots < t'_{n'} \le 1-\ell,
\qquad
t'_{j+1}-t'_j > \ell,
\]
and such that $\{a\}\times I$ admits a $c'$-separable covering with respect to the renormalization operators associated
to the pair $(K_1,\tilde K'_1)$.
This covering is given by the families $\{\mathcal F_l\}_{l=1}^{c'}$, where
\[
\mathcal F_l
:= \bigl\{x \mapsto \ell x + t'_{c'(k-1)+l} \;:\; 1 \le k \le n'/c' \bigr\},
\qquad 1 \le l \le c'.
\]

\noindent
We now define $K'_1$ as the homogeneous Cantor set generated by the subfamily
\[
\bigl\{x \mapsto \ell x + t'_j \;:\;
j = c'(k-1)+Nl,\;
1 \le k \le n'/c',\;
1 \le l \le N
\bigr\},
\]
which ensures that $\{a\}\times I$ admits a $c$-separable covering with respect to the renormalization operators
associated to the pair $(K_1,K'_1)$.
By construction, most intervals of length $\ell$ are removed, creating large gaps.
In particular, each gap of $K'_1$ has length at least $N\ell$, and hence
\[
\tau(K'_1) \le \frac{\ell}{N\ell} = \frac{1}{N}.
\]
Thus, by taking $N$ sufficiently large, $\tau(K'_1)$ can be made arbitrarily small.

\providecommand{\bysame}{\leavevmode\hbox to3em{\hrulefill}\thinspace}

\end{document}